\documentclass[a4paper]{amsart}
\usepackage{cleveref}
\usepackage{enumerate}

\usepackage{pgfplots}
\usepackage{xifthen}
\usepackage{easyvector}

\usepackage{xifthen}
\usepackage{comment}

\usepackage{marginnote}

\usepackage{xpatch}
\makeatletter
\patchcmd{\@mn@margintest}{\@tempswafalse}{\@tempswatrue}{}{}
\patchcmd{\@mn@margintest}{\@tempswafalse}{\@tempswatrue}{}{}
\makeatother

\newfont{\NUMBERS}{msbm8 scaled\magstep1}

\newcommand{\REAL}{\mbox{\NUMBERS R}}

\usepackage{ifthen}

\newcommand{\Vect}[2][]
{
  \ifthenelse{\equal{#1}{}}
  {\boldsymbol{#2}}
  {{#2}_{#1}}
}
\newcommand{\Matr}[2][]
{
  \ifthenelse{\equal{#1}{}}
  {{#2}}
  {{#2}_{#1}}
}
\newcommand{\Image}[1]{\mbox{Im}(#1)}

\newcommand{\Var}{\ensuremath{\operatorname{var}}}
\newcommand{\Err}{\ensuremath{\operatorname{err}}}

\newcommand{\Tendsto}{\rightarrow}

\newcommand{\dx}{\, dx}

\DeclareMathOperator{\Div}{\operatorname{div}}
\DeclareMathOperator{\Grad}{\nabla}
\DeclareMathOperator{\Sign}{\operatorname{sign}}

\def\XXint#1#2#3{{\setbox0=\hbox{$#1{#2#3}{\int}$ }
\vcenter{\hbox{$#2#3$ }}\kern-.6\wd0}}

\newcommand{\ABS}[2][]
{
  \ifthenelse{\equal{#1}{}}
  {\left|#2\right|}
  {\left|#2\right|_{#1}}
}
\newcommand{\NORM}[2][]
{
  \ifthenelse{\equal{#1}{}}
  {\left\|#2\right\|}
  {\left\|#2\right\|_{#1}}
}
\newcommand{\SCAL}[3][]
{
  \ifthenelse{\equal{#1}{}}
  {\langle{#2},{#3}\rangle}
  {\langle{#2},{#3}\rangle_{#1}}
}
\newcommand{\Diag}[1]{\operatorname{diag}\left(#1\right)}
\newcommand{\TimeInt}{I}
\newcommand{\Tfinal}{\infty}
\newcommand{\IdentSymb}{I}
\newcommand{\Identity}[1][]
{
  \ifthenelse{\equal{#1}{}}
  {\IdentSymb}
  {\IdentSymb_{#1}}
}

\newcommand{\Derpar}[2][]
{
  \ifthenelse{\equal{#1}{}}
  {\partial#2}
  {\partial_{#1}#2}
}
\newcommand{\Dt}[1]{\Derpar[t]{#1}}

\newcommand{\Lyap}{\mathcal{L}}

\newcommand{\Wmass}{\mathcal{M}}
\newcommand{\Ene}{\mathcal{E}}

\newcommand{\Dim}{d}
\newcommand{\Domain}{\Omega}

\newcommand{\Niter}{m}
\newcommand{\Niterpp}{\Niter+1}
\newcommand{\tstep}{k}
\newcommand{\tstepp}{\tstep+1}
\newcommand{\Deltat}[1][]{
  \ifthenelse{\equal{#1}{}}
  {\Delta t_{\tstep}}
  {\Delta t_{#1}}
  }

\newcommand{\MeshPar}{h}
\newcommand{\Cell}[1][]{\ifthenelse{\equal{#1}{}}{T}{T_{#1}}}
\newcommand{\Tsymb}{\mathcal{T}}
\newcommand{\Triang}[1][]
{
  \ifthenelse{\equal{#1}{}}
  {\Tsymb}
  {\Tsymb_{#1}}
}

\newcommand{\Vof}[2]{[#1]^{#2}}

\newcommand{\Lspacechar}{L}
\newcommand{\Lspace}[1]{\Lspacechar^{#1}}

\newcommand{\Cachar}{\mathcal{C}}

\newcommand{\Cont}[1][]{
  \ifthenelse{\equal{#1}{}}
  {\Cachar} 
  {\Cachar^{#1}}
}

\newcommand{\HolderExp}{\delta}
\newcommand{\Holder}[1][]{
  \ifthenelse{\equal{#1}{}}
  {\Cachar^{\HolderExp}}
  {\Cachar^{#1}}  
}

\newcommand{\Lip}[1]
{
  \ifthenelse{\equal{#1}{}}
  {\text{Lip}}
  {\text{Lip}_{#1}}        
}

\newcommand{\Sob}[2][]{
  \ifthenelse{\equal{#2}{}}
  {H^{#1}}
  {W^{#1,#2}}  
}

\newcommand{\DHolder}[2][]
{
  \ifthenelse{\equal{#2}{}}
  {\Cachar^{#1,\HolderExp}}   
  {\Cachar^{#1,#2}}  
}

\newcommand{\VspaceSymb}{\mathcal{V}}
\newcommand{\Vspace}[1][]{
  \ifthenelse{\equal{#1}{}}
  {\VspaceSymb_{\MeshPar}}
  {\VspaceSymb_{\MeshPar}(#1)}
}

\newcommand{\VbaseSymb}{\varphi}
\newcommand{\Vbase}[1][]{
  \ifthenelse{\equal{#1}{}}
  {\VbaseSymb}
  {\VbaseSymb_{#1}}
}
\newcommand{\WspaceSymb}{\mathcal{W}}
\newcommand{\Wspace}[1][]{
  \ifthenelse{\equal{#1}{}}
  {\WspaceSymb_{\MeshPar}}
  {\WspaceSymb_{\MeshPar}(#1)}
}

\newcommand{\WbaseSymb}{\psi}
\newcommand{\Wbase}[1][]{
  \ifthenelse{\equal{#1}{}}
  {\WbaseSymb}
  {\WbaseSymb_{#1}}
}
\newcommand{\Tdens}{\mu}

\newcommand{\OptTdens}{\Tdens^*}
\newcommand{\TdensIni}{\Tdens_0}

\newcommand{\Pot}{u}

\newcommand{\OptPot}{\Pot^*}

\newcommand{\Inode}{v}

\newcommand{\Iedge}{e}

\newcommand{\Opt}[1]{#1^{*}}
\newcommand{\Vel}{v}
\newcommand{\OptVel}{\Vel^*}
\newcommand{\Field}{\xi}

\newcommand{\Ftest}{\varphi}

\newcommand{\Cond}{\Tdens}
\newcommand{\Volt}{u}
\newcommand{\Dgrad}{\Matr{G}}

\newcommand{\Dflux}{v}
\newcommand{\Dfield}{\zeta}

\newcommand{\Forcing}{f}

\newcommand{\TolTime}{\tau_{\mbox{{\scriptsize T}}}}

\newcommand{\Graph}{G}
\newcommand{\Nodes}{V}

\newcommand{\Edges}{E}
\newcommand{\iedge}{e}

\newcommand{\BallSymb}{B}
\newcommand{\Ball}[2][]
{
  \ifthenelse{\equal{#2}{}}
  {\BallSymb_{#1}}
  {\BallSymb(#2,#1)}
}
\newcommand{\Avgint}[3][]
{
  \ifthenelse{\equal{#1}{}}
  {({#3})_{#2}}
  {(#3)_{#1,#2}}
}

\newcommand{\LF}{Lyapunov functional}

\newcommand{\OTP}{OTP}

\newcommand{\PP}[1][]
{
 \ifthenelse{\equal{#1}{full}}
  {Physarum Polychephalum}
  {PP}
}
\newcommand{\DMK}[1][]
{
 \ifthenelse{\equal{#1}{full}}
  {Dynamic Monge-Kantorovich}
  {DMK}
}
\newcommand{\PD}[1][]
{
 \ifthenelse{\equal{#1}{full}}
  {Physarum Dynamics}
  {PD}
}

\newcommand{\BP}[1][]
{
 \ifthenelse{\equal{#1}{full}}
  {Basis Pursuit}
  {BP}
}

\newcommand{\PDBP}[1][]
{
 \ifthenelse{\equal{#1}{full}}
  {Physarum Dynamics for Basis Pursuit}
  {PDBP}
}

\newcommand{\GFBP}[1][]
{
 \ifthenelse{\equal{#1}{full}}
  {Gradient Flow for Basis Pursuit}
  {GFBP}
}

\newcommand{\OTD}{OT\  density}

\newcommand{\Fnewton}{\Vect{F}}

\newcommand{\Jac}{J}

\newcommand{\Of}[1]{\left[#1\right]}

\newcommand{\PotOp}[1]{U\Of{#1}}
\newcommand{\VoltOp}[1]{\Vect{U}\Of{#1}}

\newcommand{\PotOf}[2][]
{
  \ifthenelse{\equal{#1}{}}
  {\Pot(#2)}
  {\Pot_{#1}(#2)}
}

\newcommand{\Nnn}{+}

\newcommand{\Diff}{A}
\newcommand{\Length}{w}
\newcommand{\Mlength}{\Matr{W}}
\newcommand{\Mcond}{\Matr{C}}

\newcommand{\Stiff}{S}

\newcommand{\Nnode}{n}
\newcommand{\Nedge}{m}
\newcommand{\GradVolt}{\Dgrad \!\Volt}

\newcommand{\Nnz}{k}
\newcommand{\DiffTest}{\Diff}
\newcommand{\DfluxTest}{\Dflux}
\newcommand{\ForcingTest}{\Forcing}

\newcommand{\EF}[1]{}
\newcommand{\MP}[1]{}

\newtheorem{Lemma}{Lemma} 
\newtheorem{Prop}{Proposition}

\newtheorem{Remark}{Remark}

\newcommand{\citet}{\cite}
\newcommand{\citep}{\cite}

\crefname{equation}{Equation}{Equations}
\crefname{theorem}{Theorem}{Theorems}
\crefname{chapter}{Chapter}{Chapters}
\crefname{figure}{Figure}{Figures}
\crefname{Conject}{Conjecture}{Conjectures}
\crefname{Problem}{Problem}{Problems}
\crefname{Prop}{Proposition}{Propositions}        
\crefname{Theo}{Theorem}{Theorems}
\crefname{Lemma}{Lemma}{Lemma}
\crefname{Corollary}{Corollary}{Corollaries}
\crefname{section}{Section}{Sections}

\begin{document}

\markboth{E. Facca et al.}{
  Extended Dynamic-Monge-Kantorovich equations
}

\title[
DMK for Basis pursuit
]{
  Physarum Dynamics and Optimal Transport for Basis Pursuit
}

\author{Enrico Facca, Franco Cardin and Mario Putti}
\address{
  Centro di Ricerca Matematica Ennio De Giorgi,
  Scuola Normale Superiore,
  Piazza dei Cavalieri, 3, 56126 Pisa, Italy\\
  (enrico.facca@sns.it)
}
\address{
  Department of Mathematics, University of Padua \\
  via Trieste 63, Padova, Italy\\
  \{cardin,putti\}@math.unipd.it
}

\maketitle

\begin{abstract}
  We study the connections between Physarum Dynamics and Dynamic Monge
  Kantorovich (DMK) Optimal Transport algorithms for the solution of
  Basis Pursuit problems. We show the equivalence between these two
  models and unveil their dynamic character by showing existence and
  uniqueness of the solution for all times and constructing a Lyapunov
  functional with negative Lie-derivative that drives the large-time
  convergence.
  We propose a discretization of the equation by means of a
  combination of implicit time-stepping and Newton method yielding an
  efficient and robust method for the solution of general basis
  pursuit problems.
  Several numerical experiments run on literature benchmark problems
  are used to show the accuracy, efficiency, and robustness of the
  proposed method.
\end{abstract}

\section{Introduction}

In recent years, a slime mold called \emph{Physarum Polychephalum}
(PP) captured the interest of many authors for its optimization
abilities first described
in~\cite{Nakagaki-et-al:2000,Tero-et-al:2010}.  The first attempt to
encode the optimization behavior of PP into a mathematical modeling
framework was proposed in~\citet{Tero-et-al:2007}. Here, the slime
mold body is schematized as a graph $\Graph=(\Nodes,\Edges)$ with
$\Nnode$ vertices $\Inode\in\Nodes$ and $\Nedge$ edges
$\Iedge\in\Edges$ of length $\Vect[\Iedge]{\Length}>0$. Given a time
interval $\TimeInt=[0,\Tfinal)$, denoting with $\Matr{\Diff}$ the
  adjacency matrix of $\Graph$, with
  $\Vect{\Length}=\{\Vect[\Iedge]{\Length}\}$ the vector of edge
  lengths, and with $\Vect{\Forcing}:\Nodes\mapsto\REAL$ a vector of
  forcing values such that
  $\sum_{\Inode\in\Nodes}\Vect[\Inode]{\Forcing}=0$, the \PD[full]
  (\PD) proposed by~\citet{Tero-et-al:2007} tries to find the curves
  \EF{$\TimeInt=[0,+\infty[$? Is the statement of the problem}
      $\Vect{\Cond}(t):\TimeInt\mapsto\REAL^{\Nedge}$ and
      $\Vect{\Volt}(t):\TimeInt\mapsto\REAL^{\Nnode}$ that solve the
      following equations:
\begin {subequations}
  \label{eq:pp-volt-cond} 
  \begin{align}
    \label{eq:pp-volt-cond-fick}
    &
    \Vect{\Dflux}(t)
    = \Mcond\Of{\Vect{\Cond}(t)}\Mlength^{-1}
      \Matr{\Diff}^T \Vect{\Volt}(t); \quad
      \Mcond\Of{\Vect\Cond(t)}=\Diag{\Vect\Cond(t)}; \quad  
      \Mlength=\Diag{\Vect\Length}\;, \quad   \\
    \label{eq:pp-volt-cond-cons}
    &
      \Matr{\Diff}\Vect{\Dflux}(t)=\Vect{\Forcing}\;, \\
    \label{eq:pp-volt-cond-dyn}
    &
      \Dt{\Vect{\Cond}}(t)=   
      \ABS{\Vect{\Dflux}(t)}
      -\Vect{\Cond}(t)\;,   \\
    \label{eq:pp-volt-cond-bc}
    &
      \Vect{\Cond}(0)=\hat{\Vect{\Cond}}>0\;,
  \end{align}
\end{subequations}
where the absolute value $|\cdot|$ is taken component-wise.
In~\citep{Bonifaci-et-al:2012} it was proved that the flux
$\Vect{\Dflux}(t)$ converges as $t\Tendsto\infty$ toward a steady
state equilibrium $\Opt{\Vect{\Dflux}}$ that solves the following
weighted $l_1$ minimization problem on the graph $\Graph$, also called
\emph{Optimal Transshipment Problem}:
\begin{gather}
  \label{eq:basis-pursuit}
  \min_{\Vect{\Dflux}\in\REAL^{\Nedge}}
  \NORM[1,\Mlength]{\Vect{\Dflux}}
  :=\sum_{\Iedge\in\Edges}\ABS{\Vect[\Iedge]{\Dflux}}\Length_{\Iedge}
  \quad \mbox{s.t.:}\ \Matr{\Diff}\ \Vect{\Dflux} = \Vect{\Forcing}
  \;,
\end{gather}
More recently, \citet{Straszak-Vishnoi:2016} extended the application
of the model given by~\cref{eq:pp-volt-cond} to general matrices
$\Matr{\Diff}\in \REAL^{\Nnode,\Nedge}$, arriving at a
$l_1$-minimization problem with weight $\Vect{\Length}$ (often equal
to one) related to \emph{\BP[full]} (\BP).  Moreover, the same authors
unraveled the relationship of their approach with the Iteratively
Reweighted Least Square method (IRLS), an Alternating Minimization
algorithm~\citet{Beck:2015} for the solution of $l_1$-minimizations.

The above \PD\ model inspired also the study
in~\citet{Facca-et-al:2018}, where an extension to the continuous
framework, called the \emph{\DMK[full]} problem (\DMK), was presented
and studied. In this latter study, the authors propose a conjecture
stating that the long time solution of \DMK\ is solution to the Optimal
Transport Problem (\OTP)~\cite{Villani:2008}, and in particular the
Monge-Kantorovich equations, with cost equal to the Euclidean distance,
(the so-called $\Lspace{1}$-\OTP).

The first goal of our work is to show the analogies between the
$\Lspace{1}$-\OTP\ and the \BP\ problems,
\EF{Is this really new and interesting?}
and between the continuum \DMK\ model and the discrete \PD\ model
proposed by~\citep{Straszak-Vishnoi:2016}, extending and
re-interpreting some of their results to highlight the dynamic nature
of the model.  We provide a novel and extended proof of uniqueness
and global existence of the solution to the system of equations,
\EF{Bonifaci's proof is limited to the graph case}
and show the existence of a \LF\ $\Lyap$, i.e. a functional decreasing
along solution trajectories.  In addition, we re-adapt to the discrete
case the proof given in~\cite{Facca-et-al-numeric:2018} showing that
the minimization of $\Lyap$ is equivalent to the $l_1$ minimization
problem~\cref{eq:basis-pursuit}.
\MP{Ho messo la referenza a 1.2, per capire bene cos'e' l'$l_1$ min pbl}

Our second goal is to fully exploit the dynamic nature of \PD from a
computational point of view and use an implicit (Backward Euler) time
stepping scheme combined with Newton-Raphson method to drastically
reduce the computational cost of the solution of \BP\ via \PD.

The paper is organized as follows.
In ~\cref{sec:PDtoDMK} we describe the \DMK\ model in the continuous
setting, recalling the main results and conjectures.
\Cref{sec:DMKtoBP} starts from the semi-discrete (discretized in
space) \DMK\ problem to derive the \BP\ dynamics and show global
existence and uniqueness of the solution of the resulting ODE
system. Moreover, we prove that there exist a \LF\ and that its unique
minimum (at infinite times) is the solution of an \OTP. In this
section we try to use notations that are similar to the continuous
case to highlight the similarities between the two problems, and use
the duality relationships between the gradient and divergence
operators.
The next section develops our final algorithm, introducing the
implicit time discretization (backward Euler) and Newton method.
Finally, a set of numerical results are used to show the efficiency
and robustness of our algorithm.

\section{From \PD[full] to \DMK[full]}
\label{sec:PDtoDMK}

We start this section by recalling the main results proposed
in~\cite{Facca-et-al:2018,Facca-et-al-numeric:2018} in the continuous
setting.  The \DMK\ model can be described as follows.  Consider a
convex, bounded and smooth domain $\Domain\in\REAL^{\Dim}$, a function
$\Forcing\in\Lspace{2}(\Domain)$ with $\int_{\Domain}\Forcing = 0$,
and a strictly positive function $\TdensIni$ on $\Domain$. We want to
find the pair of functions $\Pot:\TimeInt\times\Domain\mapsto\REAL$ and
$\Tdens:\TimeInt\times\Domain\mapsto\REAL$ solving: 
\begin{subequations}
  \label{eq:dmk}
  \begin{align}
    \label{eq:dmk-cons}
    &  -\Div\left(\Tdens(t,x)\Grad\Pot(t,x)\right)=\Forcing(x)\;,
    \\ 
    \label{eq:dmk-dyn}
    & \Dt{\Tdens}(t,x)=\ABS{\Tdens(t,x)\Grad\Pot(t,x)} -
    \Tdens(t,x)\;,
    \\ 
    \label{eq:dmk-bc}
    & \Tdens(0,x) = \TdensIni(x)\;,
  \end{align}
\end{subequations}
where the first equation is intended in a weak sense and is completed
with zero Neumann Boundary conditions.  In~\cite{Facca-et-al:2018} it
was conjectured that the solution of the above problem convergences as
$t\Tendsto\infty$ to a steady state configuration
$(\OptTdens,\OptPot)$ solution of the so-called Monge-Kantorovich
Equation, a PDE based formulation of the
$\Lspace{1}$-\OTP~\citep{Evans-Gangbo:1999}.

Denoting with $\PotOp{\bar{\Tdens}}$ the weak solution of
$-\Div(\bar{\Tdens}\Grad\bar{\Pot})=\Forcing$ for a given
$\bar{\Tdens}>0$, system~\cref{eq:dmk} is converted into the following
ordinary differential equation in Banach space:
\begin{subequations}
  \label{eq:dmk-ode}
  \begin{align}
    \label{eq:dmk-ode-dyn}
    & \Dt{\Tdens}(t,x) =
    \Tdens(t,x)\ABS{\Grad\PotOp{\Tdens(t,x)})} - \Tdens(t,x)\;,
    \\
    \label{eq:dmk-ode-bc}
    &\Tdens(0) = \TdensIni\;.
  \end{align}
\end{subequations}
The existence of a global solution of~\cref{eq:dmk-ode} is still an
open question and so is convergence toward equilibrium. However,
in~\cite{Facca-et-al:2018,Facca-et-al-numeric:2018} the following
results are proved:
\begin{Prop}
  \label{prop:dmk-lie}
  There exists a $\bar{t}>0$ such that~\cref{eq:dmk-ode} has a unique
  solution $\Tdens(t)$ for all $t \in [0,\bar{t}[$ that bounded from
  below by $e^{-t}\min(\TdensIni)$.  Moreover the functional
  \begin{gather}
    \label{eq:dmk-lyap}
    \Lyap(\Tdens)
    :=\Ene(\Tdens)+\Wmass(\Tdens)\;,
    \\
    \Ene(\Tdens)
    :=\sup_{\Ftest\in\Lip{}(\Domain)}
    \int_{\Domain}
    \left(
      \Forcing \Ftest-\Tdens\frac{\ABS{\Grad\Ftest}^2}{2}
    \right)
    \dx
    \nonumber
    \quad
    \Wmass(\Tdens)
    :=\frac{1}{2}\int_{\Domain}\Tdens\dx\;,
  \end{gather}
  is a \LF\ for \cref{eq:dmk-ode}, with derivative along
  the $\Tdens(t)$-trajectory given by:
  \begin{equation}
    \label{eq:dmk-lie-der}
    \frac{d}{dt}\left(\Lyap(\Tdens(t)) \right)
    =-\frac{1}{2}
    \int_{\Omega}{\Tdens(t)
      \left(\ABS{\Grad\PotOp{\Tdens(t)}}+1)\right)
      \left(\ABS{\Grad\PotOp{\Tdens(t)}}-1\right)^2
    }\dx\;.
  \end{equation}
  In addition, $\Lyap(\Tdens)$ admits a unique minimizer $\OptTdens$,
  the so-called \OTD, and the following equalities hold:
  \begin{equation}
    \label{eq:dmk-min-lyap}
    \min_{\Tdens\in\Lspace{1}(\Domain)} 
    \Lyap(\Tdens)=
    \min_{\Vel\in\Vof{\Lspace{1}(\Domain)}{\Dim}}
    \left\{
      \int_{\Omega}{\ABS{\Vel}}\dx\ :\ -\Div(\Vel)=\Forcing
    \right\}
    =\max_{\Pot\in\Lip{1}(\Domain)}
    \left\{
      \int_{\Omega}{\Pot\Forcing}\dx
    \right\}\;,
  \end{equation}
  where the unique minimizers $\OptTdens$ and $\OptVel$, and the
  non-unique maximizer $\OptPot$ are related by the following
  equation:
  \begin{equation}
    \label{eq:dmk-minima}
    \OptVel=\OptTdens\Grad\OptPot\;.
  \end{equation}
\end{Prop}
\begin{Remark}
  For a strictly positive measure $\Tdens$ the energy functional
  $\Ene(\Tdens)$ can be written equivalently as:
  \begin{equation}
    \label{eq:dmk-ene-explicit}
    \Ene(\Tdens)=
    \frac{1}{2}\int_{\Domain}\Tdens\ABS{\Grad\PotOp{\Tdens}}^2\dx.
  \end{equation}
\end{Remark}

\section{From \DMK[full] to \BP[full]}
\label{sec:DMKtoBP}

The semi-discrete version of~\cref{eq:dmk} obtained after spatial
discretization by e.g. the finite element method as done
in~\citet{Facca-et-al-numeric:2018} is related
to~\cref{eq:pp-volt-cond}. Indeed, assuming for simplicity that
matrix $\Matr{\Diff}$ in~\cref{eq:basis-pursuit}  
is of full-rank, using the matrices $\Mlength$ and
$\Mcond\Of{\Vect{\Cond}}$ defined in~\cref{eq:pp-volt-cond-fick}
and the additional matrices:
\begin{equation}
  \label{eq:matrices-1}
  \Dgrad:=\Mlength^{-1}\Matr{\Diff}^{T}
  \qquad
  \Matr{\Stiff}\Of{\Vect{\Cond}}:=\Matr{\Diff}
    \Mcond\Of{\Vect{\Cond}}
    \Mlength^{-1}\Matr{\Diff}^{T} \; ,
\end{equation}
we can rewrite~\cref{eq:pp-volt-cond} as follows:
\begin {subequations}
  \label{eq:discrete-volt-cond} 
  \begin{align}
    \label{eq:discrete-volt-cond-cons}
    \Matr{\Stiff}\Of{\Vect{\Cond}(t)}
    &\Vect{\Volt}(t)
     =\Vect{\Forcing}\;,
    \\
    \label{eq:discrete-volt-cond-dyn}
    \Dt{\Vect{\Cond}}(t)
    &= \ABS{\Vect{\Dflux}(t)}-\Vect{\Cond}(t) \\
    &= \ABS{\Mcond\Of{\Vect{\Cond}(t)}\Mlength^{-1}
      \Matr{\Diff}^T\Vect{\Volt}(t)}-\Vect{\Cond}(t) 
      = \ABS{\Mcond\Of{\Vect{\Cond}(t)}\Dgrad\Vect{\Volt}(t)}
      -\Vect{\Cond}(t) \nonumber
    \\
    \label{eq:discrete-volt-cond-bc}
    \Vect{\Cond}(0)
    &=\hat{\Vect{\Cond}}>0 \; ,
  \end{align}
\end{subequations}
where as before the absolute value of a vector $\ABS{\cdot}$ is
intended element-by-element. We can identify in
$\Matr{\Stiff}\Of{\Vect{\Cond}(t)}$ the finite element stiffness
matrix for a given $\Vect{\Cond}$ and in
$\Dgrad\Vect{\Volt}(t)=\Mlength^{-1}\Matr{\Diff}^{T}\Vect{\Volt}(t)$
the discretization of the gradient of $\PotOp{\Tdens(t)}$. This
semi-discrete algorithm is the same as the one proposed
in~\citet{Straszak-Vishnoi:2016} for the solution of the \BP[full]
problem.
\MP{E' veramente lo stesso?}
Analogously to the continuous case, given a strictly positive
$\Vect{\Cond}$, we define $\VoltOp{\Vect{\Cond}}$ as the solution of
the linear system
$\Matr{\Stiff}\Of{\Vect{\Cond}}\Vect{\Volt}=\Vect{\Forcing}$ and
rewrite~\cref{eq:discrete-volt-cond} in the following system of ODEs:
\begin {subequations}
  \label{eq:ode-cond} 
  \begin{align}
    \label{eq:ode-cond-dyn}
    &\Dt{\Vect{\Cond}}(t)=
      \Mcond\Of{\Vect{\Cond}(t)}\ABS{\Dgrad\VoltOp{\Vect{\Cond}(t)}}
      -\Vect{\Cond}(t) \;,\\
      \label{eq:ode-cond-bc-1}
    &\Vect{\Cond}(0)=\hat{\Vect{\Cond}}>0 \;.
  \end{align}
\end{subequations}
Now we are able to prove that the solution of the above problem exists
and is unique for all times and that this dynamics admits a unique
long-time equilibrium solution. Moreover we can define a \LF\ whose
minimizer is solution of a discrete optimal transport problem.
Indeed, we have the following results:
\begin{Prop}
  \label{prop:discrete}
  \Cref{eq:ode-cond} has a global unique solution for all
  $t\in\TimeInt$ and admits a
  \LF\ given by:
  \begin{gather}
    \label{eq:discrete-lyap}
    \Lyap(\Vect{\Cond}):=\Ene(\Vect{\Cond})+ \Wmass(\Vect{\Cond})\; ,
    \\
    \Ene(\Vect{\Cond})=\sup_{\Vect{\Volt}\in\REAL^{\Nnode}}
    \left\{
      \Vect{\Forcing}^T\Vect{\Volt}
      -
      \frac{1}{2}\Vect{\Volt}^T\Matr{\Stiff}\Of{\Vect{\Cond}}\Vect{\Volt}
    \right\}
    \quad
    \Wmass(\Vect{\Cond}):=\frac{1}{2}\Vect{\Cond}^T\Vect{\Length}\;.
  \end{gather}
  The derivative of $\Lyap$ along the $\Vect{\Cond}(t)$-trajectory
  takes on the form:
  \begin{equation}
    \label{eq:discrete-lie-der}
    \frac{d}{dt}\Lyap(\Vect{\Cond}(t))=
    -\frac{1}{2} \sum_{\Iedge=1}^{\Nedge}\Vect[\Iedge]{\Cond}(t)
    \left(\ABS[\Iedge]{\Vect{\GradVolt}(t)}+1\right)
    \left(\ABS[\Iedge]{\Vect{\GradVolt}(t)}-1\right)^2
    \Vect[\Iedge]{\Length} \; .
  \end{equation}
  and the following equalities hold:
  \begin{equation}
    \label{eq:discrete-min-lyap}
    \min_{\Vect{\Tdens}\in\REAL^{\Nedge}\geq 0} 
    \Lyap(\Vect{\Tdens})=
    \min_{\Vect{\Dflux}\in\REAL^{\Nedge}}
    \left\{
      \NORM[1,\Mlength]{\Vect[\Iedge]{\Dflux}}
      \ :\ 
      \Matr{\Diff}\Vect{\Dflux}=\Vect{\Forcing}
    \right\}
    =
    \max_{\Vect{\Volt}\in\REAL^{\Nnode}}
    \left\{
      \Vect{\Pot}^T\Vect{\Forcing}
      \ : \ 
      \NORM{\Dgrad\Vect{\Volt}}_{l_{\infty}}\leq 1 \; ,
    \right\}
  \end{equation}
  where the unique minimizers $\Opt{\Vect{\Cond}}$ and
  $\Opt{\Vect{\Dflux}}$, and the non-unique maximizer
  $\Opt{\Vect{\Volt}}$ are related by the following equation:
  \begin{equation}
    \label{eq:discrete-minima}
    \Opt{\Vect{\Dflux}} 
    =\Diag{\Opt{\Vect{\Cond}}}\Dgrad\Opt{\Vect{\Volt}}
    =\Mcond\Of{\Opt{\Vect{\Cond}}}\Dgrad\Opt{\Vect{\Volt}}\; .
  \end{equation}
\end{Prop}
\begin{Remark}
  As in the continuous case, for strictly positive $\Vect{\Cond}$ the
  energy functional $\Ene(\Vect{\Cond})$ can be written in analogy to
  the continuous case in the following alternative forms:
  \begin{equation}
    \label{eq:ene-explicit}
    \Ene(\Vect{\Cond})=\frac{1}{2}
    \left(\Dgrad\VoltOp{\Vect{\Cond}}\right)^T
    \Mlength\Mcond\Of{\Vect{\Cond}}\Dgrad\VoltOp{\Vect{\Cond}}
    =\frac{1}{2}\left(\VoltOp{\Vect{\Cond}}\right)^T
    \Matr{\Stiff}\Of{\Vect{\Cond}}\VoltOp{\Vect{\Cond}} \;.
  \end{equation}
\end{Remark}
The proof of the previous proposition requires the following lemma
whose proof can be found in \citet{Ekeland-Teman:1999} (see
also~\cite{Bouchitte-et-al:1997} for the continuous analogue).
\begin{Lemma}
  \label{lemma:duality-ene}
  Given $\Vect{\Cond}\in \REAL^{\Nedge}_{\Nnn}$, non-negative, and
  $\Vect{\Forcing} \in \Image{\Matr{\Diff}}$, the following equalities
  hold:
  \begin{equation}
    \Ene(\Vect{\Cond})=\sup_{\Vect{\Volt}\in \REAL^{\Nnode}}
    \left\{
      \Vect{\Forcing}^T\Vect{\Volt}
      -\frac{1}{2}\Vect{\Volt}^T\Matr{\Stiff}\Of{\Vect{\Cond}}\Vect{\Volt}
    \right\}
    =
    \inf_{\Vect{\Field}\in\REAL_{\Vect{\Cond},\Vect{\Length}}^{\Nnode}}
    \left\{  
      \frac{1}{2}
      \SCAL[\Vect{\Cond},\Mlength]{\Vect{\Dfield}}{\Vect{\Dfield}}
      \ : \ 
      \Matr{\Diff}\Mcond\Of{\Vect{\Cond}}\Vect{\Dfield}=\Vect{\Forcing}
    \right\}\;,
  \end{equation}
  where $\REAL_{\Vect{\Cond},\Vect{\Length}}^{\Nedge}$ denote
  $\REAL^{\Nedge}$ with scalar product  
  $\SCAL[\Vect{\Cond},\Mlength]{\Vect{x}}{\Vect{y}}
  := (\Mcond\Of{\Vect{\Cond}}\Mlength\Vect{x})^{T}\Vect{y}$.
\end{Lemma}
\begin{proof}[Proof of~\cref{prop:discrete}]
  We first prove local existence and uniqueness of the solution
  $\Vect{\Cond}(t)$ of~\cref{eq:ode-cond}.  For any initial data
  $\hat{\Vect{\Cond}}>0$, there exists a neighborhood
  $\Ball[\hat{\Vect{\Cond}}]{}$ where all
  $\Vect{\Cond}\in \Ball[\hat{\Vect{\Cond}}]{}$ are strictly
  positive. Let
  $\VoltOp{\cdot}:\Ball[\hat{\Vect{\Cond}}]{}\mapsto\REAL^{\Nnode}$
  be the operator defined by
  $\VoltOp{\Vect{\Cond}}=
  (\Matr{\Stiff}\Of{\Vect{\Cond}})^{-1}\Vect{\Forcing}$.  By the
  Cayley-Hamilton theorem, the vector
  $(\Matr{\Stiff}\Of{\Vect{\Cond}})^{-1}\Vect{\Forcing}$ is a
  combination of sums, products, and reciprocals of the elements of
  $\Vect{\Cond}$ and hence is infinitely differentiable
  ($\Cont[\infty]$).
  \EF{Analogous of Cohen results}
  Now, all terms on the right-hand side of~\cref{eq:ode-cond} are
  Lipschitz-continuous and we can invoke Banach-Caccioppoli fixed
  point theory to guarantee existence and uniqueness of a solution
  $\Vect{\Cond}(t)$ for $t\in [0,\bar{t}]$, with $\bar{t}$ depending
  on $\hat{\Vect{\Cond}}$ and $\Vect{\Forcing}$. From the mild
  solution of~\cref{eq:ode-cond} the lower bound
  $\Vect{\Cond}(t)\geq e^{-t}\min(\hat{\Vect{\Cond}})$ is easily
  derived, showing that $\Vect{\Cond}$ is strictly positive for
  $\hat{\Vect{\Cond}}>0$. This observation directly yields global
  existence and uniqueness of $\VoltOp{\Vect{\Cond}(t)}$ for all
  $t>0$. 

  To prove~\cref{eq:discrete-lie-der}, take the solution pair
  $(\Vect{\Volt}(t),\Vect{\Cond}(t))$ of~\cref{eq:discrete-volt-cond}
  for $t\in [0,\bar{t}[$, where
  $\Vect{\Volt}(t)=\VoltOp{\Vect{\Cond}(t)}$.  We want to compute the
  time derivative of $\Ene(\Vect{\Cond}(t))$:
  \begin{align}
    \label{eq:ene-der}
    \frac{d}{dt}\Ene(\Vect{\Cond}(t))
    &=
      \frac{1}{2}\frac{d}{dt}
      \left(
      \Vect{\Volt}^T(t)\Matr{\Stiff}\Of{\Vect{\Cond}(t)}\Vect{\Volt}(t)
      \right)\; ,
    \\
    \nonumber
    &=
      \frac{1}{2}
      \left(
      2\Vect{\Volt}^T(t)\Matr{\Stiff}\Of{\Vect{\Cond}(t)}
         \Dt{\Vect{\Volt}}(t)
      +\Vect{\Volt}^T(t)\Matr{\Stiff}\Of{\Dt{\Vect{\Cond}}(t)}
         \Vect{\Volt}(t)
      \right)\; .
  \end{align}
  Differentiating with respect to time the system
  $\Matr{\Stiff}\Of{\Vect{\Cond}(t)}\Vect{\Volt}(t)=\Vect{\Forcing}$
  we obtain:
  \begin{equation*}
    \Matr{\Stiff}[\Dt{\Vect{\Cond}}(t)]\Vect{\Volt}(t) 
    +\Matr{\Stiff}\Of{\Vect{\Cond}(t)}\Dt{\Vect{\Volt}}(t)=0\; .
  \end{equation*}
  Multiplication of both sides by 
  $\Vect{\Volt}^T(t)$ yields:
  \begin{equation*}
    \Vect{\Volt}^T(t)\Matr{\Stiff}\Of{\Vect{\Cond}(t)} 
    \Dt{\Vect{\Volt}}(t)
    =
    -\Vect{\Volt}^T(t)\Matr{\Stiff}[\Dt{\Vect{\Cond}}(t)] 
    \Vect{\Volt}(t)\;.
  \end{equation*}
  Substituting the last equation in~\cref{eq:ene-der}, we can express
  the time derivative of the energy as:
  \begin{equation*}
  \frac{d}{dt}
    \Ene(\Vect{\Cond}(t))
    =-\frac{1}{2}
    \Vect{\Volt}^T(t)\Matr{\Stiff}\Of{\Dt{\Vect{\Cond}}(t)}\Vect{\Volt}(t)
    =-\frac{1}{2}( \Dgrad\Vect{\Volt}(t))^T
    \Mlength\Mcond[\Dt{\Vect{\Cond}}(t)]\Dgrad\Vect{\Volt}(t)\; .
  \end{equation*}
  Adding the time derivative of the mass functional
  $\Wmass(\Vect{\Cond}(t))$, a few simple algebraic manipulations
  yield \cref{eq:discrete-lie-der}.
  
  Global existence and uniqueness of the solution
  vector $\Vect{\Cond}(t)$ at all times $t\in\TimeInt$ now follows.
  Indeed, the time derivative of $\Lyap(\Vect{\Cond}(t))$ is strictly
  negative and thus $\Vect{\Cond}(t)$ is forced to live for all
  times $t>0$ in a weighted $l_1$-ball of $\REAL^{\Nedge}$ with radius
  $\Lyap(\Vect{\Cond}(0))$, as the following inequality states:
  \begin{equation*}
    \NORM[{1,\Mlength}]{\Vect{\Cond}(t)}=\Wmass(\Vect{\Cond}(t))
    \leq\Ene(\Vect{\Cond}(t))+\Wmass(\Vect{\Cond}(t))
    =\Lyap(\Vect{\Cond}(t))\leq\Lyap(\Vect{\Cond}(0)) \; .
  \end{equation*}
  From this, we automatically obtain global existence
  for~\cref{eq:ode-cond} by applying well know results on maximal
  solutions of ODEs~\citep[Corollary 3.2]{Hartman:1982}.

  To prove~\cref{eq:discrete-min-lyap} we use the results of the
  duality~\cref{lemma:duality-ene} to rewrite the energy functional
  $\Ene(\Vect{\Cond})$ in the following variational form:
  \begin{align*}
    \Ene(\Vect{\Cond})
    =
    \inf_{\Vect{\Field}\in\REAL_{\Vect{\Cond}}^{\Nedge}}
    \left\{
      \frac{1}{2}\sum_{\Iedge\in\Edges}
      \Vect[\Iedge]{\Field}^2\Vect[\Iedge]{\Cond}\Vect[\Iedge]{\Length}
      \ : \
      \Matr{\Diff}\Mcond\Of{\Vect{\Cond}}\Vect{\Field} 
      =\Vect{\Forcing}
    \right\} \; .
  \end{align*}
  Then, the \LF\ $\Lyap$  can be rewritten as:
  \begin{equation*}
    \Lyap(\Vect{\Cond}) 
    = 
    \inf_{\Vect{\Field}\in\REAL^{\Nedge}_{\Vect{\Cond}}}
    \frac{1}{2}\sum_{\Iedge\in\Edges}
    \Vect[\Iedge]{\Field}^2\Vect[\Iedge]{\Cond}\Vect[\Iedge]{\Length}
    +
    \frac{1}{2}\sum_{\Iedge\in\Edges}
    \Vect[\Iedge]{\Cond}\Vect[\Iedge]{\Length} \;.
  \end{equation*}
  For any
  $\Vect{\Field}\in\REAL^{\Nedge}_{\Vect{\Cond}},
  \Vect{\Cond}\in\REAL^{\Nedge}_{+}$, by Young inequality we obtain:
  \begin{equation*}
    \sum_{\Iedge\in\Edges}
    \ABS{\Vect[\Iedge]{\Field}\Vect[\Iedge]{\Cond}}
    \Length_{\Iedge}
    =
    \sum_{\Iedge\in\Edges}
    \ABS{\Field_{\Iedge}\Cond_{\Iedge}^{1/2}\Cond_{\Iedge}^{1/2}}
    \leq
    \frac{1}{2}\sum_{\Iedge\in\Edges} 
    \Vect[\Iedge]{\Field}^{2}\Vect[\Iedge]{\Cond}\Vect[\Iedge]{\Length}
    +
    \frac{1}{2}\sum_{\Iedge\in\Edges} 
    \Vect[\Iedge]{\Cond}\Vect[\Iedge]{\Length} \; .
  \end{equation*}
  Since the above inequality holds for any
  $\Vect{\Cond}\in \REAL_+^{\Nedge}$, taking the infimum over the
  $\Vect{\Field} \in \REAL_{\Vect{\Cond}}^{\Nedge}$ yields:
  \MP{Ho sostituito il primo quadrato con un ABS, GIUSTO?}
  \begin{align*}
    &\inf_{\Vect{\Field}\in\REAL_{\Vect{\Cond}}^{\Nedge}}
      \left\{
        \sum_{\Iedge\in\Edges}
        \ABS{\Vect[\Iedge]{\Field}\Vect[\Iedge]{\Cond}}
             \Vect[\Iedge]{\Length}
        \ : \
        \Matr{\Diff}\Mcond\Of{\Vect{\Cond}} \Vect{\Field} 
        = \Vect{\Forcing}
      \right\}   
    \\
    \leq
    &\inf_{\Vect{\Field}\in\REAL_{\Vect{\Cond}}^{\Nedge}}
      \left\{
        \frac{1}{2}\sum_{\Iedge\in\Edges}
        \Field^{2}_{\Iedge}\Cond_{\Iedge}\Length_{\Iedge}
        +
        \frac{1}{2}\sum_{\Iedge\in\Edges}
        \Vect[\Iedge]{\Cond}\Length_{\Iedge}
        \ : \
        \Matr{\Diff}\Mcond\Of{\Vect{\Cond}} \Vect{\Field} 
        = 
        \Vect{\Forcing}
      \right\}
    \\
    =
    &\Lyap(\Vect{\Cond})\qquad\forall\Vect{\Cond}\in\REAL_+^{\Nedge}\; ,
  \end{align*}
  and thus, since $\Dflux_{\iedge}=\Field_{\iedge}\Cond_{\iedge}$, we
  can write:
  \begin{gather*}
    \inf_{\Vect{\Dflux}\in\REAL^{\Nedge}}
    \left\{
      \sum_{\Iedge\in\Edges}\ABS{\Dflux_{\Iedge}}\Length_{\Iedge}
      \ : \
      \Matr{\Diff}\Vect{\Dflux}=\Vect{\Forcing}
    \right\}
    \leq
    \inf_{\Vect{\Cond}\in\REAL_{+}^{\Nedge}}\Lyap(\Vect{\Cond}) \; .
  \end{gather*}
  Now, we let $\Opt{\Vect{\Dflux}},\Opt{\Vect{\Volt}}$ be solutions of
  the problems in~\cref{eq:discrete-min-lyap,eq:discrete-minima}
  ($\Opt{\Vect{\Volt}}$ is in general not unique) and consider the
  vector defined as:
  \begin{equation*}
    \Opt{\Vect[\Iedge]{\Cond}}:=\ABS{\Opt{\Vect[\Iedge]{\Dflux}}} \; .
  \end{equation*}
  We can write the following chain of equalities and inequalities:
  \MP{Cos'\`e $\Lyap_{1}$? non e' mai stata definita prima. Va sosituita
    con $\Lyap$?}
  \begin{equation*}
    \sum_{\Iedge\in\Edges}
    \ABS{\Opt{\Vect[\Iedge]{\Dflux}}}\Length_{\Iedge}
    =
    \inf_{\Vect{\Dflux}\in\REAL^{\Nedge}}
    \left\{
      \sum_{\Iedge\in\Edges}\ABS{\Vect[\Iedge]{\Dflux}}\Length_{\Iedge}
      \ : \
      \Matr{\Diff}\Vect{\Dflux}=\Vect{\Forcing}
    \right\}
    \leq
    \inf_{\Vect{\Cond}\in\REAL_{+}^{\Nedge}}\Lyap(\Vect{\Cond})
    \leq
    \Lyap_{1}(\Opt{\Vect{\Cond}})
    =
    \sum_{\Iedge\in\Edges}
    \ABS{\Opt{\Vect[\Iedge]{\Dflux}}}\Length_{\Iedge} \; ,
  \end{equation*}
  showing that $\Opt{\Vect{\Cond}}$ is a minimum of the \LF\ $\Lyap$
  and that
  $\Vect[\Iedge]{\Dflux}=\Opt{\Vect[\Iedge]{\Cond}}
  \Vect[\Iedge]{\GradVolt}\Of{\Vect{\Cond}}$ solves
  \cref{eq:basis-pursuit}.
\end{proof}

\begin{Remark}
  \MP{E' da lasciare questa?}
  We want to emphasize that, in the continuous case, the identification
  of the metric space where $\Tdens$ belongs to is a crucial step to
  establish the existence of solutions for the \DMK\ model. In fact,
  in the infinite dimensional case, given a strictly positive measure
  $\TdensIni$ there exists a measure $\Tdens$ belonging to an
  arbitrarily small ball of $\TdensIni$ that is not bounded away from
  zero (consider, e.g., the space $\Lspace{1}(\Domain)$ with the
  standard metric).  However, in the finite dimensional setting of
  interest for our current work this difficulty is completely avoided.
\end{Remark}

\section{Numerical discretization}
\label{sec:num-results}

In~\citep{Straszak-Vishnoi:2016} the \PD\ in~\cref{eq:ode-cond} is
discretized with forward Euler time stepping with fix time-step.  In
this case, the authors show that the scheme can be interpreted as an
IRLS scheme. Standard stability limitations characteristic of the
forward Euler method prevent the use of large time-steps.  In
addition, the use of a constant time-step size does not exploit the
expected exponential convergence towards steady state of the dynamical
system.  To overcome these difficulties, we generate the numerical
approximation (minimization) sequence $(\Vect{\Cond}^{\tstep})$ by
discretizing the dynamics~(\ref{eq:discrete-volt-cond}) via the
unconditionally stable backward Euler discretization with variable
step size.  The non-linear system resulting at each time step is
solved by means of an Inexact Newton-Krylov method. The latter
approach allows also the potential exploitation of the sparsity
structures of matrices $\Matr{\Diff}$, $\Mcond\Of{\Vect\Cond}$, and
$\Mlength$, increasing computational efficiency, as seen in the
numerical examples. We describe next our implementation, show some
numerical results on literature benchmarks, discussing some heuristic
explanations of the behavior of the proposed scheme.

\subsection{Backward Euler and Newton method}
\label{sec:Euler-Newton}

Discretization of~\cref{eq:discrete-volt-cond} by the implicit
backward Euler scheme leads to the following sequence of
nonlinear-algebraic equations for
$(\Vect{\Volt}^{\tstepp},\Vect{\Cond}^{\tstepp})$:
\begin{align}
  \nonumber
  &\Matr{\Stiff}\Of{\Vect{\Cond}^{\tstepp}}\Vect{\Volt}^{\tstepp}
  =\Vect{\Forcing} 
  \\[-0.5em]
  \label{ei-p1-p0}
  \\[-0.5em]
  \nonumber
  &\Vect{\Cond}^{\tstepp}
  = \Vect{\Cond}^{\tstep} +\Deltat[\tstep] 
    \left(
       \Mcond\Of{\Vect{\Cond}^{\tstepp}}\ABS{\Dgrad \Vect{\Volt}^{\tstepp}}
       -\Vect{\Cond}^{\tstepp}
    \right) \; .
\end{align}
We restate the nonlinear system to be solved at each time step $k$ as
the problem of finding the zero $\Vect{z}=(\Vect{\Volt},\Vect{\Cond})$
of the function:
\begin{equation}
  \label{eq:non-linear-cond}
  \Fnewton(\Vect{\Volt},\Vect{\Cond})=
  \begin{pmatrix}
    \Fnewton_1(\Vect{\Volt},\Vect{\Cond})\\ 
    \Fnewton_2(\Vect{\Volt},\Vect{\Cond})
  \end{pmatrix}
  =
  \begin{pmatrix}
    \Matr{\Stiff}\Of{\Vect{\Cond}}\Vect{\Volt}-\Vect{\Forcing}\\
    \Vect{\Cond} -\Deltat[\tstep] 
    \left(
       \Mcond\Of{\Vect{\Cond}}\ABS{\Dgrad\Vect{\Volt}}-\Vect{\Cond}
    \right)
    -\Vect{\Cond}^{k}  
  \end{pmatrix}
  =0 \; .
\end{equation}
Denoting with $\Niter$ the nonlinear iteration index, Newton method for
finding the zero of the above function can be written as:
\begin{align}
  \nonumber
  \Matr{J}(\Vect{z}^{\Niter})\Vect{s}
  &= -\Fnewton(\Vect{z}^{\Niter})\\[-0.5em]
  \label{newton}\\[-0.5em]
  \nonumber
  \Vect{z}^{\Niterpp} &= \Vect{z}^{\Niter} + \Vect{s} \; .
\end{align}
The Jacobian matrix is given by:
\MP{Non capisco il conto dello Jacobiano. Cos'e' il ``2'', e la
  matrice $G$? Il gradiente? Il blocco 1,2 non dovrebbe essere
  $A^TC'(\mu)W^{-1}Au$?  anche il blocco 2,1 non capisco come mai c'e'
  il trasposto. Inoltre in $D_3$ non ci dovrebbe essere $\mu$.
  Ho rifatto i conti. Devi verificarli. La matrice $D_1$ \`e la $C$!!!}
\begin{align}
  \label{eq-jacobian-pp}
  \Matr{\Jac}(\Vect{\Volt}, \Vect{\Cond}) 
  &=
  \begin{pmatrix}
    \Derpar[\Vect{\Volt}]{\Fnewton_1(\Vect{\Volt},\Vect{\Cond})} &
    \Derpar[\Vect{\Cond}]{\Fnewton_1(\Vect{\Volt},\Vect{\Cond})}
    \\
    \Derpar[\Vect{\Volt}]{\Fnewton_2(\Vect{\Volt},\Vect{\Cond})} &
    \Derpar[\Vect{\Cond}]{\Fnewton_2(\Vect{\Volt},\Vect{\Cond})}
  \end{pmatrix} \; ,\\
  &=
  \begin{pmatrix}
    \Matr{\Stiff}\Of{\Vect{\Cond}}
    &
    \Matr{\Diff}\Matr{G}\Of{\Vect{\Volt}}
    \\
    -\Deltat[\tstep]\Mcond\Of{\Vect{\Cond}}\Matr{D_1}\Of{\Vect{\Volt}}
    \Dgrad
    &
    \Matr{D_2}\Of{\Vect{\Volt}} 
  \end{pmatrix} \; ,
      \nonumber
\end{align}
where $\Matr{G}$, $\Matr{D_1}$, and $\Matr{D_2}$ are diagonal matrices
given by:
\begin{gather*}
  \Matr{G}\Of{\Vect{\Volt}}=\Diag{\Dgrad\Vect{\Volt}}
  \quad
  \Matr{D_1}\Of{\Vect{\Volt}}=\Diag{\Sign(\Vect{\Dgrad\Vect{\Volt}})}
  \quad
  \Matr{D_2}\Of{\Vect{\Volt}}
  =\Diag{\Vect{1}-\Deltat[\tstep]
  \left(
    \ABS{\Vect{\Dgrad\Vect{\Volt}}}-\Vect{1}
  \right)} \; .
\end{gather*}
Here, the $\Sign{}$ function is again intended element-by-element and
$\Vect{1}$ indicates the vector with all elements equal to one.
Assuming that all the matrices are evaluated at $\Vect{z}^{\Niter}$,
the Newton linear system can be reduced to the $\Nnode\times\Nnode$
symmetric system given by:
\begin{subequations}
  \label{eq:red-Newton-system}
  \begin{align}
    \label{eq:RN-lin-system}
    \Matr{M}\Vect{s}_1
    &:=
      \left(\Matr{\Stiff}+\Deltat[\tstep]
      \Matr{\Diff}\Matr{G}\Matr{D_2}^{-1}\Mcond\Matr{D_1}\Dgrad
      \right)s_1
      =
      \Matr{\Diff}\Matr{G}\Matr{D_2}^{-1}\Fnewton_2-\Fnewton_1 \\
    \label{eq:RN-second}
    \Vect{s}_2
    &=
      \Matr{D_2}^{-1}
      \left(
      \Deltat[\tstep]\Matr{D_2}^{-1}\Mcond\Matr{D_1}\Dgrad\Vect{s}_1
      -\Fnewton_2
      \right) \; .
  \end{align}
\end{subequations}
Since $\Vect{\Cond}$ and $\Vect{\Volt}$ appear linearly, the symmetric
matrix
$\Matr{M}$ can be rewritten as 
$\Matr{M}=\Stiff[\Vect{\Tilde{\Cond}}]$ with
$\Vect{\Tilde{\Cond}}=\{\Vect[\Iedge]{\Tilde{\Cond}}\}$ defined as:
\begin{equation*}
  \Vect[\Iedge]{\Tilde{\Cond}}=
  \Vect[\Iedge]{\Cond}
  +\Deltat[\tstep]
    \frac{\Vect[\Iedge]{\Cond}\ABS[\Iedge]{\Vect{\Dgrad\Volt}}}%
    {1-\Deltat[\tstep]\left(\ABS[\Iedge]{\Vect{\Dgrad\Volt}}-1\right)}
  =\Vect[\Iedge]{\Cond}
  \frac{1+\Deltat[\tstep]}%
  {1-\Deltat[\tstep]\left(\ABS[\Iedge]{\Vect{\Dgrad\Volt}}-1\right)} \; ,
\end{equation*}
Note that, if $\Deltat[\tstep]$ small enough, $\Matr{D_2}$ is strictly
positive and thus always invertible, and matrix $\Matr{M}$ is
symmetric and positive definite. As supported also by the numerical
evidence, in practice this requirement is always met since at large
times $\ABS[\Iedge]{\Vect{\Dgrad\Vect{\Volt}}}\approx\Vect{1}$ and we
start from $\Vect{\TdensIni}>0$.  Thus the time-step size can be
progressively increased as time advances.

\paragraph{Linear solver}
We implemented an Inexact Newton method with a Preconditioned
Conjugate Gradient (PCG) solver for the linear
system~\eqref{eq:RN-lin-system} and standard variable forcings as
suggested in~\citet{Dembo:1982}. The matrix vector product
$\Matr{M}\Vect{v}$ is performed sequentially from right to left.
Matrix $\Matr{M}$ is formed explicitly only when calculating the
preconditioner, thus allowing potential exploitation of sparsity
in matrix $\Matr{\Diff}$.

In our preliminary implementation, preconditioning is obtained by
calculating the Cholesky decomposition of $\Matr{M}$ and reusing it
until degradation of the PCG performance is detected (number of PCG
iterations larger than a specified threshold). The exit criterion is
based on the preconditioned residual~\citep{Kelley:2018}.  Note that,
as time progresses, some values of $\Vect[\Iedge]{\Cond}$ tend to zero,
drastically increasing the ill-conditioning of the linear system. To
cope with this problem, we use the spectral preconditioner developed
in~\citet{Martinez-et-al:2017}, which makes efficient use of partial
eigenspectrum information.

\paragraph{Time stepping strategy}
At each time step, we start the Newton iteration with the converged
values at the previous time step
$\Vect{z}^{0}=(\Vect{\Volt}^{\tstep},\Vect{\Cond}^{\tstep})$.  At
every time step the value of $\Deltat[\tstep]$ is doubled and a check
on the positivity of the entries in $\Matr{D_2}$ is performed. If this
check fails, $\Deltat[\tstep]$ is decreased by a factor 2.  The
initial $\Deltat[0]$ is typically unitary.  We consider that
equilibrium (or time-convergence) is achieved when:
\begin{equation}
  \label{eq:vara-cond}
 \Var(\Vect{\Cond}\Of{t^\tstep})
 =
 \frac{\NORM{\Vect{\Cond}^{\tstepp}-\Opt{\Vect{\Cond}}}}
      {\Deltat[\tstepp] \NORM{\Opt{\Vect{\Cond}}}} <\TolTime \; .
\end{equation}
In our numerical experiments we set $\TolTime=5\times 10^{-8}$.

\subsection{Numerical Experiments}

We run the first benchmark case suggested in~\citet{Yang-et-al:2010},
which considers problem~\cref{eq:basis-pursuit} with unit weights
($\Vect{\Length}=\Vect{1}$).  The $\Nnode\times\Nedge$ matrix
$\Matr{\DiffTest}$ is dense and randomly generated with elements
extracted from the standard normal distribution, normalized so that
each row has unit Euclidean norm.  The exact sparse solution
$\Vect{\Opt{\DfluxTest}}$ is randomly generated with a random support
of given dimension $\Nnz$ and entries extracted from a uniform
distribution in the interval $[-10,10]$. The right-hand side is
generated as
$\Vect{\ForcingTest}=\Matr{\DiffTest} \Vect{\Opt{\DfluxTest}}$.  We
build a sequence of four problems starting with
$(\Nnode,\Nedge,\Nnz)=(250,25000,5)\times 2^{i-1}$, $i=1,2,3,4$. Note
that the pre-imposed sparsity level $\Nnz$ of the thought solution has
no influence on the computational efficiency of our algorithm.

Simulation metrics include convergence towards steady state
($\Var(\Vect{\Cond}\Of{t^\tstep})$), relative error on the exact solution,
and constraint error in the dual \BP\ problem
(see~\cref{eq:discrete-min-lyap}), the last two defined as:
\begin{equation*}
  \Err_{\Opt{\Vect{\DfluxTest}}}(\Vect{\DfluxTest}\Of{t^{\tstep}}):=
  \frac{
    \|\Vect{\DfluxTest}\Of{t^{\tstep}}-\Opt{\Vect{\DfluxTest}}\|
  }{
    \|\Opt{\Vect{\DfluxTest}}\|
  }
  \quad
  \Err_{\mbox{Dual}}(\Vect{\Volt}\Of{t^{\tstep}})
  =\ABS{\left(\NORM[\infty]{\Matr{\Diff}^T\Vect{\Volt}\Of{t^{\tstep}}}
        -1\right)} \; .
  \end{equation*}
All the runs are conducted on a 3.4GHz Intel-I7 (1-core) CPU.

\paragraph{\PD[full]}

\begin{figure}
 \centerline{
    \includegraphics[width=0.8\textwidth,
    trim={0.0cm 6cm 0.0cm 0cm},clip
    ]{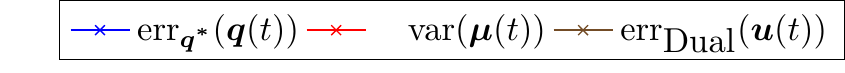}
  }
  \foreach \nrow/\nedge in {250/25000, 500/50000, 1000/100000, 2000/200000} { 
    \centerline
    {
      \includegraphics[width=0.49\textwidth,
      trim={0.0cm 0.1cm 0.0cm 0.8cm},clip
      ]{figs/runs/random_\nrow_\nedge/output/plots/errors_bp.pdf}
      \quad
      \includegraphics[width=0.48\textwidth,
      trim={0.0cm 0.1cm 0.0cm 0.8cm},clip
      ]{figs/runs/random_\nrow_\nedge/output/plots/cpu_errors_bp.pdf}
    }
  }
  \caption{Numerical results for \PD. Left column: time ($t$) evolution of
    $\Var(\Vect{\Cond}\Of{t})$,
    $\Err_{\Opt{\DfluxTest}}(\Vect{\DfluxTest}\Of{t})$,
    and  $\Err_{\mbox{Dual}}(\Vect{\Volt}\Of{t})$ for the four test
    problems. The values are sampled at every time-step $t^{\tstep}$
    (indicated with crosses) and completed by linear interpolation.
    Right column: values of the three simulation metrics as a function
    of computational time (seconds of CPU).}
  \label{fig:PD}
\end{figure}

\Cref{fig:PD} reports the results for the four different problems.
We first observe that the three metrics have a common time-evolution,
differing only by a multiplicative constant in the log-log plots in
the left column. Future studies will be devoted to analyze their
relationship so that better error estimators can be used to develop
more efficient time-step strategies.
Secondly, we would like to emphasize the distance between time-steps
that increases progressively as steady state is approach, a
consequence of our empirical time-step adaptation strategy. 

\MP{Enrico: qui sto inventando perche' non ho i dati. verifica.}

The column on the right, showing the behavior of the three metrics
with respect to CPU time, shows that some time steps are less CPU
efficient than others, as indicated by visibly smaller slopes of the
interpolation lines between consecutive metric samples in particular
towards the lower accuracy levels. To discuss this occurrence, we
focus on the first problem set (top-right figure) and look at the
16-th time step ($t_{CPU}\in[19,20]$~sec). The empirical time-stepping
strategy initially overestimated $\Deltat[16]$ and this forced a
re-assembly of the dense matrix and its Choleski factorization for a
number of times, thus drastically contributing to the loss of
computing time.  The development of problem-adapted preconditioning
strategies such as those proposed by
\citet{Fountoulakis:2013,Dassios:2015} will be considered in the
future for better performance.

\section{Conclusions}

We have presented an efficient numerical implementation of the
Physarum dynamics algorithm using a dynamic Monge-Kantorovich
perspective for the solution of general basis pursuit problems. We
show existence and uniqueness of the solution for all times and
introduce a Lyapunov functional with decreasing derivative along the
trajectories.

Our current implementation, based on combining implicit time-stepping
with Inexact Newton-Krylov methods, is shown to be accurate, robust
and efficient in solving basis pursuit benchmark problems taken from
the literature.

\section*{Acknowledgments}
This work was partially funded by the the UniPD-SID-2016 project
“Approximation and discretization of PDEs on Manifolds for
Environmental Modeling” and by the EU-H2020 project
``GEOEssential-Essential Variables workflows for resource efficiency
and environmental management'', project of ``The European Network for
Observing our Changing Planet (ERA-PLANET)'', GA 689443.

\bibliographystyle{siam}
\bibliography{Strings,biblio}

\bibliographystyle{hsiam}
\bibliography{Strings,biblio}

\end{document}